\documentclass[reqno,12pt]{amsart}
\usepackage{amssymb}
\usepackage{amsmath}
\usepackage{mathrsfs}
\usepackage{amsmath,amssymb}
\usepackage[colorlinks = true,psdextra,unicode]{hyperref}
\usepackage{paralist}
\usepackage{graphics} 
\usepackage{epsfig} 
\usepackage[colorlinks = true]{hyperref}
\hypersetup{urlcolor = red, citecolor = blue}
\usepackage[top = 1in,bottom = 1in,left = 1in,right = 1in]{geometry}
\usepackage{cite}
\usepackage{esint}
\usepackage{verbatim}

\usepackage[pagewise]{lineno}
\usepackage{graphicx}
\usepackage{subcaption}

\usepackage{float}  
\usepackage{caption}
\usepackage{geometry}
\allowdisplaybreaks[4]
\newtheorem{thm}{Theorem}[section]

\newtheorem{lem}[thm]{Lemma}
\newtheorem{prop}{Proposition}[section]
\theoremstyle{definition}

\numberwithin{equation}{section}
\begin{document}
\title [Quasilinear two-species chemotaxis system with two chemicals]
{Finite-time blow-up in a quasilinear two-species chemotaxis system with two chemicals }
\author[Cai]{Mingzhang Cai}
\address{M. Cai: School of Mathematics, Southeast University, Nanjing 211189, P. R. China}
\email{220242012@seu.edu.cn}

\author[Li]{Yuxiang Li$^{\star}$}
\thanks{$^{\star}$Corresponding author}
\address{Y. Li: School of Mathematics, Southeast University, Nanjing 211189, P. R. China}
\email{lieyx@seu.edu.cn}

\author[Zeng]{Ziyue Zeng}%
\address{Z. Zeng: School of Mathematics, Southeast University, Nanjing 211189, P. R. China}
\email{ziyzzy@163.com}


\thanks{Supported in part by National Natural Science Foundation of China (No. 12271092, No. 11671079) and the Jiangsu Provincial Scientific Research Center of Applied Mathematics (No. BK20233002)}

\begin{abstract}
This paper investigates the finite-time blow-up phenomena to a quasilinear two-species chemotaxis system with two chemicals
\begin{align}\tag{$\star$}
	\begin{cases}
		u_t = \nabla \cdot \left(D_1(u) \nabla u\right) -  \nabla \cdot \left(u \nabla v\right), & x \in \Omega, \ t > 0, \\ 
		0 = \Delta v - \mu_2 + w, \quad \mu_2=\fint_{\Omega}w, & x \in \Omega, \ t > 0, \\ 
		w_t = \nabla \cdot \left(D_2(w) \nabla w\right)  -  \nabla \cdot \left(w \nabla z\right), & x \in \Omega, \ t > 0, \\ 
		0 = \Delta z - \mu_1 + u, \quad \mu_1=\fint_{\Omega}u, & x \in \Omega, \ t > 0, \\
		\frac{\partial u}{\partial \nu} = \frac{\partial v}{\partial \nu} = \frac{\partial w}{\partial \nu} = \frac{\partial z}{\partial \nu} = 0, & x \in \partial \Omega, \ t > 0, \\
		u(x, 0) = u_0(x), \quad w(x, 0) = w_0(x), & x \in \Omega,
	\end{cases}
\end{align}
where $\Omega \subset \mathbb{R}^n$ $(n \geqslant 3)$ is a smoothly bounded domain. The nonlinear diffusion functions \( D_1(s) \) and \( D_2(s) \) are of the following forms:
\begin{align*}
	D_1(s)\simeq  s^{m_1-1} \quad \text{and}\quad D_2(s) \simeq  s^{m_2-1}, \quad m_1,m_2> 1
\end{align*}
for $s\geqslant 1$. 

For the classical  two-species chemotaxis system with two chemicals (i.e. the second and fourth equations are replaced by $0 = \Delta v - v + w$ and $0 = \Delta z - z + u$ ), Zhong [J. Math. Anal. Appl., 500 (2021),
Paper No. 125130, pp. 22.] showed that the system possesses a globally bounded classical solution in the case that \[
m_1 + m_2 < \max\left\{m_1m_2 + \frac{2m_1}{ n},\ m_1m_2 + \frac{2m_2 }{ n}\right\}.
\] 

Complementing the boundedness result, we prove that the system ($\star$) admits solutions that blow up in finite time, if
\[
m_1 + m_2 > \max\left\{ m_1m_2 + \frac{2m_1}{ n},\ m_1m_2 + \frac{2m_2}{ n}\right\}
\]
with $n\geqslant 3$.

\end{abstract}

\maketitle

\section{Introduction}
This work is devoted to an investigation of two-species chemotaxis system with two chemicals proposed by Tao and Winkler\cite{2015-DCDSSB-TaoWinkler}
\begin{align}\label{eq1.1}
	\begin{cases}
		u_t = \nabla \cdot \left(D_1(u) \nabla u\right) -  \nabla \cdot \left(u \nabla v\right), & x \in \Omega, \ t > 0, \\ 
		0 = \Delta v - \mu_2 + w, \quad \mu_2=\fint_{\Omega}w, & x \in \Omega, \ t > 0, \\ 
		w_t = \nabla \cdot \left(D_2(w) \nabla w\right)  -  \nabla \cdot \left(w \nabla z\right), & x \in \Omega, \ t > 0, \\ 
		0 = \Delta z - \mu_1 + u, \quad \mu_1=\fint_{\Omega}u, & x \in \Omega, \ t > 0, \\
		\frac{\partial u}{\partial \nu} = \frac{\partial v}{\partial \nu} = \frac{\partial w}{\partial \nu} = \frac{\partial z}{\partial \nu} = 0, & x \in \partial \Omega, \ t > 0, \\
		u(x, 0) = u_0(x), \quad w(x, 0) = w_0(x), & x \in \Omega,
	\end{cases}
\end{align}
	where $\Omega \subset \mathbb{R}^n$ $(n \geqslant 3)$ is a smoothly bounded domain. The sensitivity functions, denoted as \( D_1(s) \) and \( D_2(s) \), exhibit asymptotic behavior of the form:
	\begin{align*}
		D_1(s)\simeq  s^{m_1-1} \quad \text{and}\quad D_2(s) \simeq  s^{m_2-1}, \quad m_1,m_2> 1
	\end{align*}
for $s\geqslant 1$. Complementing the boundedness result developed in \cite{2021-JMAA-Zhong}, we prove in this paper that system (\ref{eq1.1}) admits solutions that blow up in finite time.

The one-species chemotaxis system with one chemical \cite{1971-JTB-KellerSegel,1971-Jotb-KellerSegel, 2002-CAMQ-PainterHillen} is described as follows: 
\begin{align}\label{eq1.1.2}
	\begin{cases}
		u_t=\nabla \cdot(D(u) \nabla u)-\nabla \cdot(S(u) \nabla v), & x \in \Omega, t>0, \\ 
		\tau v_t=\Delta v-v+u , & x \in \Omega, t>0, \\
		\frac{\partial u}{\partial \nu} = \frac{\partial v}{\partial \nu} = 0, & x \in \partial \Omega, \ t > 0, \\
		u(x, 0) = u_0(x), \quad v(x, 0) = v_0(x), & x \in \Omega.
	\end{cases}
\end{align}
Considering that the diffusion of the chemical substance is faster than the random motion
of cells, J\"ager and Luckhaus\cite{1992-TAMS-JaegerLuckhaus} proposed the Keller-Segel system in the form
\begin{align}\label{eq1.1.1}
	\begin{cases}
		u_t=\nabla \cdot(D(u) \nabla u)-\nabla \cdot(S(u) \nabla v), & x \in \Omega, t>0, \\ 
		0=\Delta v-\mu+u, \quad \mu=\fint_{\Omega}u, & x \in \Omega, t>0 , \\
		\frac{\partial u}{\partial \nu} = \frac{\partial v}{\partial \nu} = 0, & x \in \partial \Omega, \ t > 0, \\
		u(x, 0) = u_0(x), \quad v(x, 0) = v_0(x), & x \in \Omega.
	\end{cases}
\end{align}
The critical mass $8\pi$ for $n=2$ was established by Nagai \cite{1995-AMSA-Nagai} for the radially symmetric system (\ref{eq1.1.2}) and (\ref{eq1.1.1}) with $D(u)=1$ and $S(u)=u$. Specifically, when \(m < 8\pi\), the corresponding solution remains global and bounded; however, if \(m > 8\pi\) and \(\int_{\Omega} u_0|x|^2 \mathrm{d}x\) is sufficiently small, there exist solutions that blow up in finite time. Subsequently, Nagai \cite{2001-JIA-Nagai} removed the radial symmetry assumption and showed that finite-time blow-up occurs when either \(q \in \Omega\) with \(m > 8\pi\), or \(q \in \partial \Omega\) with \(m > 4\pi\), provided that \(\int_{\Omega} u_0|x-q|^2 \mathrm{d}x\) is sufficiently small.
When considering more general forms \(D(u) = (1+u)^{-p}\) and \(S(u) = u(1+u)^{q-1}\), a critical blow-up exponent \(\frac{2}{n}\) is identified for system (\ref{eq1.1.1}). Winkler and Djie \cite{2010-NA-WinklerDjie} proved that if \(p+q > \frac{2}{n}\) and \(q > 0\), system (\ref{eq1.1.1}) admits radially symmetric solutions that blow up in finite time, whereas solutions remain globally bounded if \(p+q < \frac{2}{n}\). Similar results exist for the system (\ref{eq1.1.2}) have been established in \cite{2020-DCDSSS-Lankeit,2019-JDE-Winkler,2010-MMAS-Winkler,2012-JDE-TaoWinkler,2025-CVPDE-CaoFuest,2014-AAM-CieslakStinner,2012-JDE-CieslakStinner,2015-JDE-CieslakStinner}.

In contrast to classical systems, the chemotaxis system with indirect signal production can possess some distinct features concerning the blow-up or global existence of solutions
\begin{align}\label{indirect-01}
\begin{cases}
	u_t = \nabla \cdot \big(D(u)\nabla u - S(u)\nabla v\big), & x \in \Omega, \, t > 0, \\
	v_t = \Delta v - a_1 v + b_1 w, & x \in \Omega, \, t > 0, \\
	w_t = \Delta w - a_2 w + b_2 u, & x \in \Omega, \, t > 0,\\
	\frac{\partial u}{\partial \nu} = \frac{\partial v}{\partial \nu} = \frac{\partial w}{\partial \nu} = 0, & x \in \partial \Omega, \, t > 0,\\
	(u(x,0), v(x,0), w(x,0)) = (u_0(x), v_0(x), w_0(x)), & x \in \Omega.
\end{cases}
\end{align}
When $D(u)=1$ and $S(u)=u$, Fujie and Senba \cite{2017-JDE-FujieSenba,2019-JDE-FujieSenba} deduced a four-dimensional critical mass phenomenon for the classical solutions. It was shown that solutions are globally bounded for $n=4$ in the radially symmetric case provided that $\int_{\Omega} u_0 < 64 \pi^2$. Conversely, there exist initial data satisfying $\int_{\Omega} u_0 \in\left(64 \pi^2, \infty\right) \backslash 64 \pi^2 \mathbb{N}$ such that the solutions of (\ref{indirect-01}) under the mixed boundary conditions blow up. 
When \(D(u) = (1+u)^{-p}\) and \(S(u) = u(1+u)^{q-1}\), Ding and Wang \cite{2019-DCDSSB-Dingwang} proved that the system possess global bounded classical solutions if $q+p < \min\{1+\frac{2}{n},\frac{4}{n}\}$. For the Jäger-Luckhaus variant of the system (\ref{indirect-01}) (i.e. the second equation is modified to $\Delta v-\int_0 w /|\Omega|+w=0$), Tao and Winkler\cite{2025-JDE-TaoWinkler} found that there exist finite-time blow-up solutions if $p+q>\frac{4}{n}$ and $q>\frac{2}{n}$ with $n\geqslant 3$, under the radially symmetric condition. Mao and Li \cite{2026-JMAA-MaoLi} demonstrated that if $p+q > \frac{4}{n}$ and $q<\frac{2}{n}$ for $n\geqslant4$, radially symmetric initial data with large negative energy lead to solutions that blow up in infinite time.

The two-species chemotaxis system with two chemicals 
\begin{align}\label{eq1.1.3}
	\begin{cases}
		u_t=\nabla \cdot\left(D_1(u) \nabla u\right)-\nabla \cdot\left(S_1(u) \nabla v\right), & x \in \Omega, t>0, \\ 
		0=\Delta v-v+w, & x \in \Omega, t>0, \\ 
		w_t=\nabla \cdot\left(D_2(w) \nabla w\right)- \nabla \cdot\left(S_2(w) \nabla z\right), & x \in \Omega, t>0, \\ 
		0=\Delta z-z+u, & x \in \Omega, t>0, \\
		\frac{\partial u}{\partial \nu} = \frac{\partial v}{\partial \nu} = \frac{\partial w}{\partial \nu} = \frac{\partial z}{\partial \nu} = 0, & x \in \partial \Omega, \ t > 0, \\
		u(x, 0) = u_0(x), \quad w(x, 0) = w_0(x), & x \in \Omega,
	\end{cases}
\end{align}
exhibits phenomena similar to those of system (\ref{eq1.1.1}) and (\ref{indirect-01}). When $D_i(\xi)=1$ and $S_i(\xi)=\xi$ for $i=1,2$, a critical mass curve has recently been identified. Let $m_1:=\int_{\Omega} u_0(x) \mathrm{d} x$ and $m_2:=\int_{\Omega} w_0(x) \mathrm{d} x$. It has been shown that the solutions of (\ref{eq1.1.3}) blow up in finite time when $m_1 m_2-2 \pi\left(m_1+m_2\right)>0$ and $\int_{\Omega} u_0(x)\left|x-x_0\right|^2\mathrm{~d} x$, $\int_{\Omega} w_0(x)\left|x-x_0\right|^2\mathrm{~d} x$ are sufficiently small \cite{2018-N-YuWangZheng}. Conversely, if $m_1 m_2-2 \pi\left(m_1+m_2\right)<0$, all solutions remain globally bounded \cite{2024-NARWA-YuXueHuZhao}. For the system (\ref{eq1.1.3}) with $D_i(u)\simeq (u+1)^{p_i-1}$ and $S_i(u) \simeq u(1+u)^{q_i-1}$, Zheng \cite{2017-TMNA-Zheng} proved that all solutions are globally bounded under the conditions $q_1<\frac{2}{n}+p_1-1$ and $q_2<\frac{2}{n}+ p_2-1$. Subsequently, Zhong \cite{2021-JMAA-Zhong} studied the case $q_1=q_2=1$ and demonstrated that global classical solutions exist if either $p_1 p_2+\frac{2p_1}{n}>p_1+p_2$ or $p_1 p_2+\frac{2p_2}{n}>p_1+p_2$. For system (\ref{eq1.1.3}) with $p_i\equiv 1$, Zeng and Li \cite{2025--ZengLi-1} derived a critical blow-up curve (i.e. $q_1+q_2-\frac{4}{n}=\max\Big\{\big(q_1-\frac{2}{n}\big)q_2,\big(q_2-\frac{2}{n}\big)q_1\Big\}$ in the square $(0,\frac{4}{n}) \times (0,\frac{4}{n})$) characterizing global boundedness and finite-time blow-up. For system (\ref{eq1.1.3}) with $p_2\equiv q_2\equiv 1$, Zeng and Li \cite{2025--ZengLi} found two critical blow-up lines (i.e $q_1-(p_1-1)=2-\frac{n}{2}$ and $q_1=1-\frac{n}{2}$).

$\mathbf{Main\ results.}$ We assume that
\begin{align}\label{eq1.2}
	\begin{aligned}
		&D_1\text{ and }D_2\ \text{belong to } C^3([0, \infty))\ \text {with } D_1,D_2>0 \text{ in } [0,+\infty)
	\end{aligned}
\end{align}
and 
\begin{align}\label{eq1.4}
		D_1(s) \leqslant k_1 s^{m_1-1}, \quad  s \geqslant 1,
	\end{align}
	as well as
	\begin{align}\label{eq1.5}
		D_2(s) \leqslant k_2 s^{m_2-1}, \quad  s \geqslant 1,
	\end{align}
	with some $k_1,k_2 > 0$ and $m_1,m_2> 1$.
Suppose that 
\begin{align}\label{eq1.3}
	u_0, w_0 \in  W^{1,\infty}({\Omega})\ \text{ are radially symmetric and} \ u_0, w_0 \geqslant 0.
\end{align}

The local existence and uniqueness of classical solutions to system \eqref{eq1.1} are established in Proposition~\ref{local} by adapting the arguments from \cite{2010-NA-WinklerDjie,2015-DCDSSB-TaoWinkler,2025--ZengLi-1}.
 
\begin{prop}\label{local}
	Let $\Omega \subset \mathbb{R}^n$ $(n \geqslant 1)$ be a smoothly bounded domain. Suppose that $D_1,D_2$ satisfy $(\ref{eq1.2})$ and $u_0, w_0$ satisfy $(\ref{eq1.3})$. There exist $T_{\max } \in(0, \infty]$ and a uniquely determined pair $(u,v,w,z)$ of functions 
	\begin{align*}
		\begin{aligned}
			& u \in C^0\left(\overline{\Omega} \times\left[0, T_{\max }\right)\right) \cap C^{2,1}\left(\overline{\Omega} \times\left(0, T_{\max }\right)\right), \\
			& v \in C^{2,0}\left(\overline{\Omega} \times\left[0, T_{\max }\right)\right), \\
			& w \in C^0\left(\overline{\Omega} \times\left[0, T_{\max }\right)\right) \cap C^{2,1}\left(\overline{\Omega} \times\left(0, T_{\max }\right)\right), \\
			& z \in C^{2,0}\left(\overline{\Omega} \times\left[0, T_{\max }\right)\right),
		\end{aligned}
	\end{align*}
with $u \geqslant 0$, $w \geqslant 0$ in $\Omega \times (0,T_{\max})$ and $\int_{\Omega} v(\cdot,t) \mathrm{~d}x=0$, $\int_{\Omega} z(\cdot,t) \mathrm{~d}x=0$ for all $t \in (0,T_{\max})$, such that $(u,v,w,z)$ solves \eqref{eq1.1} classically in $\Omega \times\left(0, T_{\max }\right)$. In addition, the following blow-up criterion holds
	\begin{align}\label{exten}
		\text { if } T_{\max }<\infty \text {, then } \limsup _{t \nearrow T_{\max }}\left(\|u(\cdot, t)\|_{L^{\infty}(\Omega)}+\|w(\cdot, t)\|_{L^{\infty}(\Omega)}\right)=\infty \text {. }
	\end{align}
	Moreover,
	\begin{align*}
		\int_\Omega u(t) \mathrm{~d}x=\int_\Omega u_0 \mathrm{~d}x \ \text{ and } \ \int_\Omega w(t) \mathrm{~d}x= \int_\Omega w_0 \mathrm{~d}x,
		\quad t\in(0,T_{\max}).
	\end{align*}
\end{prop}

We now state our main results. The following theorem is concerned with the blow-up of solutions to the system (\ref{eq1.1}).
\begin{thm}\label{thm1_1}
	Let $\Omega=B_R(0) \subset \mathbb{R}^n$ $(n \geqslant 3)$ with some $R>0$ and let
\begin{align}\label{eq1.6}
	m_1+m_2 > \max\left\{m_1m_2+\frac{2m_1}{n},m_1m_2+\frac{2m_2}{n}\right\}.
\end{align}
Suppose that $D_1,D_2$ satisfy $(\ref{eq1.2})$-$(\ref{eq1.5})$.
Then there exist functions $\hat{M}_1(r), \hat{M}_2(r) \in C^0([0, R])$ such that, whenever the radially symmetric initial data $u_0$, $w_0$ satisfy $(\ref{eq1.3})$ and
	\begin{align}\label{eq1.8}
		\int_{B_r(0)} u_0 \mathrm{~d}x \geqslant \hat{M}_1(r),\quad \int_{B_r(0)} w_0 \mathrm{~d}x \geqslant \hat{M}_2(r),\quad  r\in(0,R),
	\end{align}
	the corresponding solution of \eqref{eq1.1} blows up in finite time.
\end{thm}
We organize the rest of the paper as follows. In Section~\ref{comparison}, we establish a weak comparison principle. Subsequently, Section~\ref{subsolution} is devoted to demonstrating the existence of finite-time blow-up solutions.

\section{Comparison principle}\label{comparison}
     We introduce the mass distribution functions
	\begin{align}\label{eq2.1}
		U(s,t):=\int^{s^\frac{1}{n}}_0 r^{n-1}u(r,t)\mathrm{~d}r\ \ \text{ and }\ \ W(s,t):=\int^{s^\frac{1}{n}}_0 r^{n-1}w(r,t)\mathrm{~d}r,
	\end{align}
	for $s \in\left[0, R^n\right]$ and $ t \in\left[0, T_{\max }\right)$. 
	Denote
	\begin{align}\label{eq2.2}
		\mu^{\star} := \max\left\{\fint_{\Omega} u_0 \mathrm{d}x, \fint_{\Omega} w_0 \mathrm{d}x \right\},
	\end{align}
	and
	\begin{align}\label{eq2.3}
		\mu_{\star} := \min\left\{\fint_{\Omega} u_0 \mathrm{d}x, \fint_{\Omega} w_0 \mathrm{d}x \right\}.
	\end{align}
	For any $T > 0$ and $\varphi, \psi\in C^1\left(\left[0, R^n\right] \times[0, T)\right)$ such that $\varphi_s, \psi_s \geqslant 0$ in $(0,R^n) \times (0,T)$, and $\varphi(\cdot,t) \in W_{l o c}^{2, \infty}\left(\left(0, R^n\right)\right)$, $\psi(\cdot,t) \in W_{l o c}^{2, \infty}\left(\left(0, R^n\right)\right)$ for all $t \in (0,T)$, define
	\begin{align}\label{eq2.4}
		\begin{cases}
			\begin{aligned}
				&\mathcal{P}[\varphi, \psi](s, t):=\varphi_t-n^2 s^{2-\frac{2}{n}} \varphi_{s s} D_1\left(n \varphi_s\right)-n \varphi_s \cdot\left(\psi-\frac{\mu^{\star} s}{n}\right) \\
				& \mathcal{Q}[\varphi, \psi](s, t):=\psi_t-n^2 s^{2-\frac{2}{n}} \psi_{s s} D_2\left(n \psi_s\right)-n \psi_s \cdot\left(\varphi-\frac{\mu^{\star} s}{n}\right)
			\end{aligned}
		\end{cases}
	\end{align}
	for $t \in(0, T)$ and a.e. $s \in\left(0, R^n\right)$.  \eqref{eq1.1} and \eqref{eq2.4}  yield the following Dirichlet parabolic system
	\begin{align}\label{eq2.5}
		\begin{split}
			\begin{cases}
				\mathcal{P}[U, W](s, t) \geqslant 0,\quad &s\in(0,R^n), t\in(0,T_{\max}), \\
				\mathcal{Q}[U, W](s, t) \geqslant 0, \quad &s\in(0,R^n), t\in(0,T_{\max}),\\
				U(0,t)=W(0,t)=0,\quad &t\in(0,T_{\max}),\\
				U(R^n,t)=\frac{\mu_{1}R^n}{n}\geqslant\frac{\mu_{\star}R^n}{n},\ W(R^n,t)=\frac{\mu_{2}R^n}{n}\geqslant\frac{\mu_{\star}R^n}{n},\quad &t\in(0,T_{\max}),\\
				U(s,0)=\int^{s^\frac{1}{n}}_0 r^{n-1}u_0(r,t)\mathrm{d}r,\quad &s\in(0,R^n),\\
				W(s,0)=\int^{s^\frac{1}{n}}_0 r^{n-1}w_0(r,t)\mathrm{d}r,\quad &s\in(0,R^n).
			\end{cases}
		\end{split}
	\end{align}

	%%%%%%%%%%%%%%%%%%%%%%%%%%%%%%%%%%%%%%%%%%%%%%%%%%%%%%%%%%%%%%%%%%%%%%%%%%%%%%%%%%%%%%%%%%%%%%%%%%%%%%%%%%%%%%%%%%%%%%%%%%%%%%%%%%%%%%%%%%%%%%%%%%%%%%%%%%%%%%%%%%%%%%%
An essential element for our argument is the following variant of the parabolic comparison principle, which is similar to \cite{2025--ZengLi-1}. We omit the proof.
	%%%%%%%%%%%%%%%%%%%%%%%%%%%%%%%%%%%%%%%%%%%%%%%%%%%%%%%%%%%%%%%%%%%%%%% 
	\begin{lem}\label{lem2.1}
		Let \( T > 0 \), and assume that \( D_1 \) and \( D_2 \) satisfy $(\ref{eq1.2})$. Suppose that $\underline{U}, \overline{U}, \underline{W}, \overline{W} \in C^1\left(\left[0, R^n\right] \times[0, T)\right)$ such that for all \( t \in (0, T) \),  
		$$\underline{U}(\cdot, t), \overline{U}(\cdot, t), \underline{W}(\cdot, t), \overline{W}(\cdot, t) \in W_{l o c}^{2, \infty}\left(\left(0, R^n\right)\right).$$
		Moreover, assume that  
		$$\underline{U}_s, \overline{U}_s, \underline{W}_s, \overline{W}_s\geqslant0, \quad (s,t) \in (0,R^n) \times (0,T).$$ 
		If for all $t \in(0, T)$ and a.e. $s \in\left(0, R^n\right)$, we have
		\begin{align}\label{eq2.6}
			\begin{array}{ll}
				\mathcal{P}[\underline{U}, \underline{W}](s, t) \leqslant 0, & \mathcal{P}[\overline{U}, \overline{W}](s, t) \geqslant 0, \\
				\mathcal{Q}[\underline{U}, \underline{W}](s, t) \leqslant 0, & \mathcal{Q}[\overline{U}, \overline{W}](s, t) \geqslant 0,
			\end{array}
		\end{align}
		and for $s \in [0,R^n]$, we have
	\begin{align}\label{eq2.8}
		\underline{U}(s, 0) \leqslant \overline{U}(s, 0),
		\quad   \underline{W}(s, 0) \leqslant \overline{W}(s, 0),
	\end{align}
		as well as for all $t \in[0, T)$, we have
		\begin{align}\label{eq2.7}
			\begin{array}{lll}
				\underline{U}(0, t) \leqslant \overline{U}(0, t), & \underline{U}\left(R^n, t\right) \leqslant \overline{U}\left(R^n, t\right), \\
				\underline{W}(0, t) \leqslant \overline{W}(0, t), & \underline{W}\left(R^n, t\right) \leqslant \overline{W}\left(R^n, t\right),
			\end{array}
		\end{align}
		then
		\begin{align}\label{eq2.9}
			\underline{U}(s, t) \leqslant \overline{U}(s, t), \quad
			\underline{W}(s, t) \leqslant \overline{W}(s, t), 
			\quad  (s,t) \in\left[0, R^n\right] \times [0, T).
		\end{align}
	\end{lem}
	%%%%%%%%%%%%%%%%%%%%%%%%%%%%%%%%%%%%%%%%%%%%%%%%%%%%%%%%%%%%%%%%%%%%%%%%%

\section{ Construction of subsolutions}\label{subsolution}
We follow the construction of subsolution in \cite{2025-JDE-TaoWinkler}.
In order to construct subsolutions $\underline{U}, \underline{W}$, we shall first select the parameters $\alpha, \beta$ and $\delta$.

\begin{lem}\label{alphabeta}
	Let $n \geqslant 3$. Suppose that $m_1,m_2> 1$ satisfy $(\ref{eq1.6})$.  Then there exist $\alpha,\beta,\delta \in (0,1)$ such that
	\begin{align}\label{delta1}
		\alpha+\delta-1<0
	\end{align}
and
	\begin{align}\label{delta2}
	\beta+\delta-1<0,
    \end{align}
as well as
	\begin{align}\label{pa}
		(m_1-1)(1-\alpha)+\beta-1+\frac{2}{n}<0
	\end{align}
and
	\begin{align}\label{qb}
		(m_2-1)(1-\beta)+\alpha-1+\frac{2}{n}<0.
	\end{align}
\end{lem}
\begin{proof}
When $\delta \rightarrow 0$, we have $\alpha+\delta-1 \rightarrow \alpha-1<0$ and $\beta+\delta-1 \rightarrow \beta-1<0$ for all $\alpha,\beta \in (0,1)$. Thus, we can find $\delta_{\star} \in (0,\frac{1}{2})$ such that \eqref{delta1} and \eqref{delta2} hold for all $\alpha,\beta \in (0,1)$ and $\delta \in (0,\delta_{\star})$.

In the following, we divide the domain $(1,\infty) \times (1,\infty)$ into four disjoint regions
\begin{align*}
&S_1=\Big\{(m_1, m_2)\in (1,\infty) \times (1,\infty) \mid m_1<2-\frac{2}{n},\quad m_2<2-\frac{2}{n}\Big\}, \\
&S_2=\Big\{(m_1, m_2)\in (1,\infty) \times (1,\infty) \mid m_1\geqslant 2-\frac{2}{n},\quad m_2\geqslant 2-\frac{2}{n}\Big\},\\
&S_3=\Big\{(m_1, m_2)\in (1,\infty) \times (1,\infty) \mid m_1\geqslant 2-\frac{2}{n},\quad m_2< 2-\frac{2}{n}\Big\},\\ 
&S_4=\Big\{(m_1, m_2)\in (1,\infty) \times (1,\infty) \mid m_1< 2-\frac{2}{n},\quad m_2\geqslant 2-\frac{2}{n}\Big\}.
\end{align*}
\textbf{Case 1:} $(m_1, m_2)\in S_1$. Thus, we have
	\begin{align*}
		\frac{1}{m_1}+\frac{1}{m_2}-\frac{2}{m_1n}
		=\frac{n-2}{m_1n}+\frac{1}{m_2} > \frac{n-2}{2n-2}+\frac{n}{2n-2}=1,
	\end{align*}
which implies $m_1+m_2 > m_1m_2+\frac{2m_2}{n}$. Similarly, we can easily check $m_1+m_2 > m_1m_2+\frac{2m_2}{n}$. Thus, for all $(m_1, m_2)\in S_1$, (\ref{eq1.6}) holds.

For all $(m_1, m_2)\in S_1$, we have $(m_1-1)(1-\alpha)+\beta-1+\frac{2}{n} \rightarrow m_1-2+\frac{2}{n}<0$ and $	(m_2-1)(1-\beta)+\alpha-1+\frac{2}{n} \rightarrow m_2-2+\frac{2}{n}<0$ as $(\alpha, \beta) \rightarrow (0,0)$. Then we can find $\alpha_{1} \in (0, \frac{1}{2})$ and $\beta_{1} \in (0, \frac{1}{2})$ such that (\ref{pa}) and (\ref{qb}) hold for all $\alpha \in (0,\alpha_1)$ and $\beta \in (0,\beta_1)$.  

\textbf{Case 2:} $(m_1, m_2)\in S_2$. Since 
\begin{align*}
 \frac{1}{m_1}+\frac{1}{m_2}-\frac{2}{m_1n}
 =\frac{n-2}{m_1n}+\frac{1}{m_2} \leqslant \frac{n-2}{2n-2}+\frac{n}{2n-2}=1,
\end{align*}
we find $m_1+m_2 \leqslant m_1m_2+\frac{2m_2}{n}$, which implies (\ref{eq1.6}) is not satisfied for any $(m_1, m_2)\in S_2$.

\textbf{Case 3:} $(m_1, m_2)\in S_3$.
Denote
\begin{align*} 
S=& \left\{(m_1, m_2) \mid m_1+m_2 > \max\left\{m_1m_2+\frac{2m_1}{n},m_1m_2+\frac{2m_2}{n}\right\} \right\} \cap S_3 \\
=& \left\{(m_1, m_2) \mid m_1\geqslant 2-\frac{2}{n}, m_2>1, m_1+m_2 > m_1m_2+\frac{2m_1}{n} \right\}. 
\end{align*}
Thus, for all $(m_1, m_2)\in S$, we can fix $({m_1}^{\star}, {m_2}^{\star})\in S$ satisfying $m_1<{m_1}^{\star}$ and $ m_2<{m_2}^{\star}$. Take
$$\alpha=1+\frac{\frac{2}{n} {m_2}^{\star}}{({m_1}^{\star}-1)({m_2}^{\star}-1)-1}, \quad \beta=1+\frac{\frac{2}{n} {m_1}^{\star}}{({m_1}^{\star}-1)({m_2}^{\star}-1)-1} .$$
Then we can easily check
\begin{align}\label{=0-1}
 ({m_1}^{\star}-1)(1-\alpha)+\beta-1+\frac{2}{n}=0,
 \end{align}
  and 
\begin{align}\label{=0-2} ({m_2}^{\star}-1)(1-\beta)+\alpha-1+\frac{2}{n}=0.
\end{align}
Owing to  $({m_1}^{\star}, {m_2}^{\star})\in S$, we have  
$$
	{m_1}^{\star}+{m_2}^{\star}-{m_1}^{\star}{m_2}^{\star}>\frac{2{m_2}^{\star}}{n},
$$
which implies 
$$ ({m_1}^{\star}-1)({m_2}^{\star}-1)-1={m_1}^{\star}{m_2}^{\star}-{m_1}^{\star}-{m_2}^{\star}
		< -\frac{2{m_2}^{\star}}{n}<0 .
$$
Then we deduce that
\begin{align}\label{a01}
 -1<\frac{\frac{2}{n} {m_2}^{\star}}{({m_1}^{\star}-1)({m_2}^{\star}-1)-1}<0. 
\end{align}
According to  $({m_1}^{\star}, {m_2}^{\star})\in S$, we have  
$$		{m_1}^{\star}+{m_2}^{\star}-{m_1}^{\star}{m_2}^{\star}>\frac{2{m_1}^{\star}}{n}. 
$$ 
Similarly, we obtain
\begin{align}\label{b01}
	-1< \frac{\frac{2}{n} {m_1}^{\star}}{({m_1}^{\star}-1)({m_2}^{\star}-1)-1} <0.
\end{align}
\eqref{a01} and \eqref{b01} confirm that $\alpha,\beta\in (0,1)$. Then,  using (\ref{=0-1}) and (\ref{=0-2}), along with $m_1<{m_1}^{\star}$ and $ m_2<{m_2}^{\star}$, we deduce that, for all $(m_1, m_2)\in S$, there exist  $\alpha,\beta \in (0,1)$ such that \eqref{pa} and \eqref{qb} hold.

\textbf{Case 4:} $(m_1, m_2)\in S_4$. The symmetry of $m_1$ and $m_2$
in \eqref{eq1.6} implies that a result similar to Case 3 can be obtained.
\end{proof}

Let $\alpha, \beta \in (0,1)$ as determined in Lemma \ref{alphabeta}. For all $T>0$ and $y \in C^1([0, T))$ with $y(t)>\frac{1}{R^n}$ for all $t \in (0,T)$, we introduce
\begin{align}\label{phi}
	& \hat{U}(s, t) =
	\begin{cases}
		l y^{1 - \alpha}(t) s, & t \in [0, T), s \in \left[0, \frac{1}{y(t)}\right], \\
		l \alpha^{-\alpha}  \left(s - \frac{1 - \alpha}{y(t)}\right)^\alpha, & t \in [0, T), s \in \left(\frac{1}{y(t)}, R^n\right],
	\end{cases} 
\end{align}
\begin{align}\label{psi}
	& \hat{W}(s, t) =
	\begin{cases}
		l y^{1 - \beta}(t) s, & t \in [0, T), s \in \left[0, \frac{1}{y(t)}\right], \\
		l \beta^{-\beta}  \left(s - \frac{1 - \beta}{y(t)}\right)^\beta, & t \in [0, T), s \in \left(\frac{1}{y(t)}, R^n\right].
	\end{cases}
\end{align}
where
\begin{align}\label{ldef}
	l = \frac{\mu_{\star} R^n}{n\mathrm{e}^{\frac{1}{\mathrm{e}}}(R^n + 1)},
\end{align}
and $\mu_{\star}$ is defined in \eqref{eq2.3}. 
Therefore, by straightforward computations we can derive
\begin{align}\label{phit}
	& \hat{U}_t(s, t) =
	\begin{cases}
		l (1 - \alpha) y^{-\alpha}(t)y^{\prime}(t) s, & t \in (0, T), s \in \left(0, \frac{1}{y(t)}\right), \\
		l \alpha^{1 - \alpha} (1 - \alpha)  \left(s - \frac{1 - \alpha}{y(t)}\right)^{\alpha - 1} \frac{y^{\prime}(t)}{y^2(t)}, & t \in [0, T), s \in \left(\frac{1}{y(t)}, R^n\right),
	\end{cases} 
\end{align}
\begin{align}\label{psit}
	& \hat{W}_t(s, t) =
	\begin{cases}
		l (1 - \beta) y^{-\beta}(t)y^{\prime}(t) s, & t \in (0, T), s \in \left(0, \frac{1}{y(t)}\right), \\
		l \beta^{1 - \beta} (1 - \beta)  \left(s - \frac{1 - \beta}{y(t)}\right)^{\beta - 1} \frac{y^{\prime}(t)}{y^2(t)}, & t \in (0, T), s \in \left(\frac{1}{y(t)}, R^n\right),
	\end{cases}
\end{align}
and
\begin{align}\label{phis}
	& \hat{U}_s(s, t) =
	\begin{cases}
		l y^{1 - \alpha}(t), & t \in (0, T), s \in \left(0, \frac{1}{y(t)}\right), \\
		l \alpha^{1 - \alpha}  \left(s - \frac{1 - \alpha}{y(t)}\right)^{\alpha - 1}, & t \in (0, T), s \in \left(\frac{1}{y(t)}, R^n\right),
	\end{cases} 
\end{align}
\begin{align}\label{psis}
	& \hat{W}_s(s, t) =
	\begin{cases}
		l y^{1 - \beta}(t), & t \in (0, T), s \in \left(0, \frac{1}{y(t)}\right), \\
		l \beta^{1 - \beta}  \left(s - \frac{1 - \beta}{y(t)}\right)^{\beta - 1}, & t \in (0, T), s \in \left(\frac{1}{y(t)}, R^n\right),
	\end{cases}
\end{align}
as well as
\begin{align}\label{phiss}
	& \hat{U}_{ss}(s, t) =
	\begin{cases}
		0, & t \in (0, T), s \in \left(0, \frac{1}{y(t)}\right), \\
		l \alpha^{1 - \alpha} (\alpha - 1) \cdot \left(s - \frac{1 - \alpha}{y(t)}\right)^{\alpha - 2}, & t \in (0, T), s \in \left(\frac{1}{y(t)}, R^n\right),
	\end{cases} 
\end{align}
\begin{align}\label{psiss}
	& \hat{W}_{ss}(s, t) =
	\begin{cases}
		0, & t \in [0, T), s \in \left(0, \frac{1}{y(t)}\right), \\
		l \beta^{1 - \beta} ( \beta-1)  \left(s - \frac{1 - \beta}{y(t)}\right)^{\beta - 2}, & t \in (0, T), s \in \left(\frac{1}{y(t)}, R^n\right).
	\end{cases}
\end{align}
Let $\theta>1$ be large enough, and we furthermore introduce our comparison functions
\begin{equation}\label{eq3_7}
	\begin{cases}
		\begin{array}{ll}
			\underline{U}(s, t) := \mathrm{e}^{-\theta t} \hat{U}(s, t), & t \in [0, T), s \in \left[0, R^n\right],  \\
			\underline{W}(s, t) := \mathrm{e}^{-\theta t} \hat{W}(s, t), &t \in [0, T), s \in \left[0, R^n\right].
		\end{array}
	\end{cases}
\end{equation}
Since $\hat{U}, \hat{W}$ belong to $C^1([0, R^n] \times [0, T)) \cap C^0([0, T); W^{2, \infty}((0, R^n))) $, We know that
$$\underline{U}, \underline{W} \in C^1([0, R^n] \times [0, T)) \cap C^0([0, T); W^{2, \infty}((0, R^n))).$$ 
We shall check $\underline{U}$, $\underline{W}$ are subslutions to (\ref{eq1.1}).
The proof of this fact is divided into three parts. 

Let us
start with the inner region.
\begin{lem}\label{lem3.1}
Suppose that $m_1,m_2>1$ satisfy \eqref{eq1.6}. Given $\alpha, \beta, \delta \in (0,1)$ as provided by Lemma \ref{alphabeta}, we can find $y_0$ satisfying 
	\begin{align}\label{eq3_8}
		y_0 > y_{\star}:=\max\left\{ 1, \left(\frac{2\mu^{\star}\mathrm{e}}{nl}\right)^{\frac{1}{1-\beta}}, 
\left(\frac{2\mu^{\star}\mathrm{e}}{nl}\right)^{\frac{1}{1-\alpha}}, \frac{1}{R^n}  \right\}.
	\end{align}
Assume that $T>0$ and $y(t) \in C^1([0, T))$ satisfies
	\begin{align}\label{y1}
		\left\{\begin{array}{l}
			0 \leqslant y^{\prime}(t)\leqslant 
			\min\big\{\frac{nl}{2\mathrm{e}^{2}(1-\alpha)}  , \frac{nl}{2\mathrm{e}^{2}(1-\beta)}  \big\}y^{1+\delta}(t), \quad t\in (0,T), \\
			y(0) \geqslant y_0,
		\end{array}\right.
	\end{align}
	then, for any $\theta>0$, the functions $\underline{U}$ and $\underline{W}$ from $(\ref{eq3_7})$ satisfy
	\begin{align*}
		\mathcal{P}[\underline{U}, \underline{W}](s, t) \leqslant 0, \quad \mathcal{Q}[\underline{U}, \underline{W}](s, t) \leqslant 0, 
	\end{align*}
	for all $t \in (0,T) \cap \left(0, \frac{1}{\theta}\right)$ and $s \in \left(0, \frac{1}{y(t)}\right)$.
\end{lem}
\begin{proof} 
	According to $y(t)\geqslant y_0 > \frac{1}{R^n}$, we derive that $\frac{1}{y(t)} < R^n$. 
From the definitions of $\underline{U}$ and $\underline{W}$ that
\begin{align}
	\mathcal{P}[\underline{U}, \underline{W}](s, t) &= \underline{U}_t - n^2 s^{2 - \frac{2}{n}} \underline{U}_{ss} D_1\left(n \underline{U}_s\right) - n \underline{U}_s \cdot \Big(\underline{W} - \frac{\mu^{\star} s}{n}\Big)
	\nonumber \\
	&= -\theta \mathrm{e}^{-\theta t} l y^{1 - \alpha}(t) s +
	\mathrm{e}^{-\theta t} l (1 - \alpha) y^{-\alpha}(t) y^{\prime}(t) s \nonumber \\
	&\quad - n \mathrm{e}^{-\theta t} l y^{1 - \alpha}(t) \Big(\mathrm{e}^{-\theta t} l y^{1 - \beta}(t) s - \frac{\mu^{\star} s}{n}\Big) \nonumber \\
	&\leqslant \mathrm{e}^{-\theta t} l (1 - \alpha) y^{-\alpha}(t) y^{\prime}(t) s
	- n \mathrm{e}^{-\theta t} l y^{1 - \alpha}(t) \Big(\mathrm{e}^{-\theta t} l y^{1 - \beta}(t) s - \frac{\mu^{\star} s}{n}\Big) \label{eq3_10}
\end{align}
for all $t \in (0,T) \cap \left(0, \frac{1}{\theta}\right)$ and $s \in \left(0, \frac{1}{y(t)}\right)$. Due to $t \in (0,T) \cap \left(0, \frac{1}{\theta}\right)$, we have
	\begin{align}\label{thetat}
		\theta t<1, \quad t\in (0,T)\cap \Big(0, \frac{1}{\theta}\Big).
	\end{align}
By the monotonicity of $y(t)$ and the second restriction in \eqref{eq3_8}, together with (\ref{thetat}), we derive that
\begin{align*}
   \frac{1}{2}	\mathrm{e}^{-\theta t} l y^{1 - \beta}(t) s - \frac{\mu^{\star} s}{n} 
	>  \frac{\mathrm{e}^{-1} l y_0^{1 - \beta} s}{2} 
	-\frac{\mu^{\star} s}{n}
	> 0.
\end{align*}
Therefore we can deduce that
\begin{align*}
	\mathcal{P}[\underline{U}, \underline{W}](s, t)&\leqslant l (1 - \alpha) y^{-\alpha}(t) y^{\prime}(t) s-\frac{n\mathrm{e}^{-2} l^2 y^{2-\alpha - \beta}(t) s}{2}\\
	&=l (1 - \alpha) y^{-\alpha}(t)  s\left(y^{\prime}(t)-\frac{nl}{2\mathrm{e}^{2}(1-\alpha)}y^{2-\beta}(t)\right)
\end{align*}
for all $t \in (0,T) \cap \left(0, \frac{1}{\theta}\right)$ and $s \in \left(0, \frac{1}{y(t)}\right)$. By $y(t)\geqslant y_0>1$, together with (\ref{delta2}) and (\ref{y1}), we infer that
\begin{align*}
	\mathcal{P}[\underline{U}, \underline{W}](s, t)
	\leqslant&\, l (1 - \alpha) y^{-\alpha}(t)  s
	\Big(y^{\prime}(t)-\frac{nl}{2\mathrm{e}^{2}(1-\alpha)}   y^{1+\delta}(t)  \Big)  
	\leqslant 0
\end{align*}
for all $t \in (0,T) \cap \left(0, \frac{1}{\theta}\right)$ and $s \in \left(0, \frac{1}{y(t)}\right)$.
Owing to the symmetry of the constructed subsolutions \(\underline{U}\) and \(\underline{W}\), and the differential operators \(\mathcal{P}\) and \(\mathcal{Q}\), it is possible to extend the argument. By applying the third restriction in \eqref{eq3_8}, in conjunction with (\ref{y1}), we can establish that
\begin{align*}
	\mathcal{Q}[\underline{U}, \underline{W}](s, t) \leqslant 0,
\end{align*}
for all $t \in (0,T) \cap \left(0, \frac{1}{\theta}\right)$ and $s \in \left(0, \frac{1}{y(t)}\right)$.  This completes the proof.
\end{proof}
%%%%%%%%%%%%%%%%%%%%%%%%%%%%%%%%%%%%%%%%%%%%%%%%%%%%%%%%%%%%%%%%
Let
\begin{align} \label{c1}
	c_{1} = \max\left\{\frac{\beta}{\alpha}, 1\right\} , 
\quad c_{2} = \min\left\{\frac{\beta}{\alpha}, 1\right\},
\end{align}
where $\alpha, \beta \in (0,1)$ as determined in Lemma \ref{alphabeta}.
The following lemma establishes the non-positivity of \( \mathcal{P}[\underline{U}, \underline{W}](s, t) \) and \( \mathcal{Q}[\underline{U}, \underline{W}](s, t) \) in the transition region \( (\frac{1}{y(t)}, s_{0}] \), where \( s_{0} \) is a fixed sufficiently small constant satisfying
\( s_{0} < R^n \) and
	
	\begin{align}\label{sstar1}
		{s_{0}}^{1-\alpha}<\frac{\alpha^{1-\alpha}nl}{\mathrm{e}},
		\quad {s_{0}}^{1-\alpha} < \frac{\alpha^{1-\alpha}nl}{2\mu^{\star}\mathrm{e}},
	\end{align}
	and
	\begin{align}\label{sstar2}
		s_{0}^{1-\beta}< \frac{\beta^{1-\beta}nl}{ \mathrm{e}}, \quad
		{s_{0}}^{1-\beta} < \frac{\beta^{1-\beta}nl}{2\mu^{\star}\mathrm{e}},
	\end{align}
	as well as
	\begin{align}\label{sstar3}
	\left\{\begin{array}{l}
		\frac{4\beta^{\beta}\mathrm{e}^{2}\alpha^{\delta-1}}
		{c_{2}^{\beta} nl }
		-{s_{0}}^{\beta+\delta-1} \leqslant 0, \\
\frac{4\beta^{\beta}\mathrm{e}^{2}k_1 n^{m_1}l^{m_1-2}\alpha^{(m_1-1)(1-\alpha)-2+\frac{2}{n}}}
			{c_{2}^{\beta} }
			-{s_{0}}^{(m_1-1)(1-\alpha)+\beta-1+\frac{2}{n}} \leqslant 0,
	\end{array}\right.
	\end{align}
	and
	\begin{align}\label{sstar4}
	\left\{\begin{array}{l}
	\frac{4c_{1}^{\alpha}\alpha^{\alpha}\mathrm{e}^{2}\beta^{\delta-1}}
	{ nl }
	-{s_{0}}^{\alpha+\delta-1} \leqslant 0, \\
	4c_{1}^{\alpha}\alpha^{\alpha}\mathrm{e}^{2}k_2 n^{m_2}l^{m_2-2}\beta^{(m_2-1)(1-\beta)-2+\frac{2}{n}}
	-{s_{0}}^{(m_2-1)(1-\beta)+\alpha-1+\frac{2}{n}} \leqslant 0.
	\end{array}\right.	
\end{align} 
Actually, the choice of $\alpha,\beta,\delta \in (0,1)$ as determined in Lemma \ref{alphabeta} guarantees the existence of $s_{0}>0$ satisfying (\ref{sstar1})-(\ref{sstar4}).

\begin{lem}\label{lem3.3}
	Suppose that $m_1,m_2>1$ satisfy \eqref{eq1.6}. Let $\alpha, \beta, \delta \in (0,1)$ as provided by Lemma \ref{alphabeta} and $s_0$ satisfying $(\ref{sstar1})$-$(\ref{sstar4})$. Assume that $T>0$ and $y(t) \in C^1([0, T))$ satisfy
	\begin{align}\label{y2}
		\left\{\begin{array}{l}
		0\leqslant	y^{\prime}(t) \leqslant y^{1+\delta}(t), \quad t \in (0,T), \\
			y(0) \geqslant y_0,
		\end{array}\right.
	\end{align}
	where
	\begin{align}\label{eq3_14}
		{y_0}>\frac{1}{s_{0}}.
	\end{align}
	Then, for any $\theta>0$, the functions $\underline{U}$ and $\underline{W}$ from $(\ref{eq3_7})$ satisfy
	\begin{align}\label{mid}
		\mathcal{P}[\underline{U}, \underline{W}](s, t) \leqslant 0, \quad  \mathcal{Q}[\underline{U}, \underline{W}](s, t) \leqslant 0,
	\end{align}
	for all $t \in (0,T) \cap \left(0, \frac{1}{\theta}\right)$ and $s \in \big(\frac{1}{y(t)},s_{0}\big]$.
\end{lem}
\begin{proof}
	 \eqref{eq3_14} implies that \( \frac{1}{y(t)} < \frac{1}{y_0} < s_{0} \), which guarantees that the interval \( \big( \frac{1}{y(t)}, s_{0} \big] \) is non-empty. 
According to the definitions of $\underline{U}$, $\underline{W}$ and $\mathcal{P}$, we have
\begin{align*}
	\begin{aligned}
		\mathcal{P}[\underline{U}, \underline{W}](s, t)
		&=\underline{U}_t-n^2 s^{2-\frac{2}{n}} \underline{U}_{s s}
		D_1\left(n \underline{U}_s\right)-n \underline{U}_s \cdot\Big(\underline{W}-\frac{\mu^{\star} s}{n}\Big)\\
		&=-\theta \mathrm{e}^{-\theta t}\alpha^{-\alpha}l\Big(s-\frac{1-\alpha}{y(t)}\Big)^\alpha
		+\mathrm{e}^{-\theta t}\alpha^{1-\alpha} l(1-\alpha)\Big(s-\frac{1-\alpha}{y(t)}\Big)^{\alpha-1}\frac{y^{\prime}(t)}{y^2(t)}\\
		&\quad +\mathrm{e}^{-\theta t}n^2s^{2-\frac{2}{n}}\alpha^{1-\alpha} l (1-\alpha)\Big(s-\frac{1-\alpha}{y(t)}\Big)^{\alpha-2}D_1\left(n\mathrm{e}^{-\theta t}\alpha^{1-\alpha}l\Big(s-\frac{1-\alpha}{y(t)}\Big)^{\alpha-1}\right)\\
		&\quad -n\mathrm{e}^{-\theta t}\alpha^{1-\alpha}l\Big(s-\frac{1-\alpha}{y(t)}\Big)^{\alpha-1}
		\left(\mathrm{e}^{-\theta t}\beta^{-\beta}l\Big(s-\frac{1-\beta}{y(t)}\Big)^\beta-\frac{\mu^{\star}s}{n}\right).
	\end{aligned}
\end{align*}
Since \( s > \frac{1}{y(t)} \) and \( \delta < 1 \), it follows from Lemma \ref{alphabeta} that for given \( \alpha, \beta \in (0,1) \), we deduce that	
\begin{align}\label{s1}
	\beta s<s-\frac{1-\beta}{y(t)},
	\quad \alpha s<s-\frac{1-\alpha}{y(t)},
\end{align}
and
\begin{align}\label{s5}
	y^{\delta-1}(t)<{\alpha}^{\delta-1}\Big(s-\frac{1-\alpha}{y(t)}\Big)^{1-\delta},
	\quad y^{\delta-1}(t)<{\beta}^{\delta-1}\Big(s-\frac{1-\beta}{y(t)}\Big)^{1-\delta}.
\end{align}
Due to \eqref{s1} and $n\geqslant3$, we can obtain 
\begin{align}\label{s4}
	s^{2-\frac{2}{n}}<\alpha^{\frac{2}{n}-2}{\Big(s-\frac{1-\alpha}{y(t)}\Big)}^{2-\frac{2}{n}},
	\quad
	s^{2-\frac{2}{n}}<\beta^{\frac{2}{n}-2}{\Big(s-\frac{1-\beta}{y(t)}\Big)}^{2-\frac{2}{n}}.
\end{align} 
By $\theta t<1$, along with the first restrictions in \eqref{s1}-\eqref{s4}, we deduce
\begin{align}\label{eq3_15}
	\begin{aligned}
	&\mathcal{P}[\underline{U}, \underline{W}](s, t)\\
	\leqslant&\, \alpha^{1-\alpha} l\Big(s-\frac{1-\alpha}{y(t)}\Big)^{\alpha-1}\cdot y^{\delta-1}(t)\\
	&\, +n^2s^{2-\frac{2}{n}}\alpha^{1-\alpha}l \Big(s-\frac{1-\alpha}{y(t)}\Big)^{\alpha-2}D_1\left(n\mathrm{e}^{-\theta t}\alpha^{1-\alpha}l\Big(s-\frac{1-\alpha}{y(t)}\Big)^{\alpha-1}\right)\\
	&\, -n\mathrm{e}^{-\theta t}\alpha^{1-\alpha}l\Big(s-\frac{1-\alpha}{y(t)}\Big)^{\alpha-1}
	\left(\mathrm{e}^{-\theta t}\beta^{-\beta}l\Big(s-\frac{1-\beta}{y(t)}\Big)^\beta-\frac{\mu^{\star}s}{n}\right)\\
	\leqslant& \, l\alpha^{\delta-\alpha}\Big(s-\frac{1-\alpha}{y(t)}\Big)^{\alpha-\delta}  
	+n^2 l\alpha^{\frac{2}{n}-1-\alpha} \Big(s-\frac{1-\alpha}{y(t)}\Big)^{\alpha-\frac{2}{n}}D_1\left(n\mathrm{e}^{-\theta t}\alpha^{1-\alpha}l\Big(s-\frac{1-\alpha}{y(t)}\Big)^{\alpha-1}\right)\\
	&\, -n\mathrm{e}^{-1}\alpha^{1-\alpha}l\Big(s-\frac{1-\alpha}{y(t)}\Big)^{\alpha-1}
	\left(\mathrm{e}^{-1}\beta^{-\beta}l\Big(s-\frac{1-\beta}{y(t)}\Big)^\beta-\frac{\mu^{\star}}{n\beta}\Big(s
	-\frac{1-\beta}{y(t)}\Big)\right).
	\end{aligned}
\end{align}
Owing to the second restriction in (\ref{sstar2}), we have 
\begin{align*}
\frac{\beta^{-\beta}l}{2\mathrm{e}}-\frac{\mu^{\star}}{n\beta}{s_{0}}^{1-\beta} >0.
\end{align*}
Thus, combining this with $0< \beta <1$ , we deduce that
\begin{align}\label{eq3_16}
	 &\frac{1}{2}\mathrm{e}^{-1}\beta^{-\beta}l\Big(s-\frac{1-\beta}{y(t)}\Big)^\beta-\frac{\mu^{\star}}{n\beta}\Big(s-\frac{1-\beta}{y(t)}\Big) \nonumber \\
	=& \, \left(\frac{\beta^{-\beta}l}{2\mathrm{e}}-\frac{\mu^{\star}}{n\beta}\Big(s-\frac{1-\beta}{y(t)}\Big)^{1-\beta}\right)\Big(s-\frac{1-\beta}{y(t)}\Big)^\beta \nonumber \\
	\geqslant &\, \Big(\frac{\beta^{-\beta}l}{2\mathrm{e}}-\frac{\mu^{\star}}{n\beta}{s}^{1-\beta}\Big)\Big(s-\frac{1-\beta}{y(t)}\Big)^\beta \nonumber\\
	\geqslant &\, \Big(\frac{\beta^{-\beta}l}{2\mathrm{e}}-\frac{\mu^{\star}}{n\beta}{s_{0}}^{1-\beta}\Big)\Big(s-\frac{1-\beta}{y(t)}\Big)^\beta \nonumber\\
	\geqslant&\, 0 .
\end{align}
For the case $\alpha \leqslant \beta$, owing to $s>\frac{1}{y(t)}$, we have 
\begin{align*}
	\frac{\alpha}{\beta} \cdot \Big(s-\frac{1-\beta}{y(t)}\Big)
	\leqslant  s-\frac{1-\alpha}{y(t)} \leqslant s-\frac{1-\beta}{y(t)}.
\end{align*}
For the case $\alpha > \beta$, we also have
\begin{align*}
	\frac{\beta}{\alpha} \cdot \Big(s-\frac{1-\alpha}{y(t)}\Big) \leqslant s-\frac{1-\beta}{y(t)}
	\leqslant s-\frac{1-\alpha}{y(t)}.
\end{align*}
Using the definitions of $c_1$ and $c_2$ in (\ref{c1}), we have 
\begin{align}\label{alphabeta-1}
	c_{2}\Big(s-\frac{1-\alpha}{y(t)}\Big) 
	\leqslant s-\frac{1-\beta}{y(t)}
	\leqslant c_{1}\Big(s-\frac{1-\alpha}{y(t)}\Big).
\end{align}
By \eqref{eq3_16}-\eqref{alphabeta-1}, it can be inferred that
\begin{align}\label{ag}
		&n\mathrm{e}^{-1}\alpha^{1-\alpha}l\Big(s-\frac{1-\alpha}{y(t)}\Big)^{\alpha-1}
		\cdot\left(\mathrm{e}^{-1}\beta^{-\beta}l\Big(s-\frac{1-\beta}{y(t)}\Big)^\beta
		-\frac{\mu^{\star}}{n\beta}\Big(s-\frac{1-\beta}{y(t)}\Big)\right) \nonumber\\
		\geqslant &\,n\mathrm{e}^{-1}\alpha^{1-\alpha}l\Big(s-\frac{1-\alpha}{y(t)}\Big)^{\alpha-1}
		\cdot\frac{\beta^{-\beta}l}{2\mathrm{e}}\Big(s-\frac{1-\beta}{y(t)}\Big)^\beta\nonumber\\
		\geqslant &\, n\mathrm{e}^{-1}\alpha^{1-\alpha}l\Big(s-\frac{1-\alpha}{y(t)}\Big)^{\alpha-1}
		\cdot\frac{c_{2}^{\beta}\beta^{-\beta}l}{2\mathrm{e}}\Big(s-\frac{1-\alpha}{y(t)}\Big)^\beta\nonumber\\
		\geqslant&\,
		\frac{c_{2}^{\beta}\beta^{-\beta}nl^2\alpha^{1-\alpha}}{2\mathrm{e}^2}\Big(s-\frac{1-\alpha}{y(t)}\Big)^{\alpha+\beta-1}.
\end{align}
It follows from the first restriction in (\ref{sstar1}) that
\begin{align*}
	n\mathrm{e}^{-\theta t}\alpha^{1-\alpha}l\Big(s-\frac{1-\alpha}{y(t)}\Big)^{\alpha-1}
	>\frac{\alpha^{1-\alpha}nl}{ \mathrm{e}} s^{\alpha-1}
	\geqslant\frac{\alpha^{1-\alpha}nl}{ \mathrm{e}} s_{0}^{\alpha-1}>1.
\end{align*}
Thus, combining this with \eqref{eq1.4}, and substituting \eqref{ag} into \eqref{eq3_15}, we can deduce that
\begin{align*}
	\begin{aligned}
		\mathcal{P}[\underline{U}, \underline{W}](s, t)
		\leqslant& \, k_1n^{m_1+1} l^{m_1}\alpha^{(m_1-1)(1-\alpha)+\frac{2}{n}-1-\alpha} \Big(s-\frac{1-\alpha}{y(t)}\Big)^{(m_1-1)(\alpha-1)+\alpha-\frac{2}{n}} \\
		&	+l\alpha^{\delta-\alpha}\Big(s-\frac{1-\alpha}{y(t)}\Big)^{\alpha-\delta} 
			-\frac{c_{2}^{\beta}\beta^{-\beta}nl^2\alpha^{1-\alpha}}{2\mathrm{e}^2}\Big(s-\frac{1-\alpha}{y(t)}\Big)^{\alpha+\beta-1}.
	\end{aligned}
\end{align*}
By the choice of $\alpha$ and $\beta$ in Lemma \ref{alphabeta}, we have $\beta+\delta-1<0$ and $(m_1-1)(1-\alpha)+\beta-1+2/n<0$, which implies that
\begin{align}\label{p}
	\mathcal{P}[\underline{U}, \underline{W}](s, t)  
	\leqslant& \,\frac{c_{2}^{\beta}\beta^{-\beta}nl^2\alpha^{1-\alpha}}{4\mathrm{e}^2}\Big(s-\frac{1-\alpha}{y(t)}\Big)^{\alpha-\delta} 		\left(\frac{4\beta^{\beta}\mathrm{e}^{2}\alpha^{\delta-1}}
	{c_{2}^{\beta} nl }
	-{s}^{\beta+\delta-1}\right) \nonumber \\   
	&\,+\frac{c_{2}^{\beta}\beta^{-\beta}nl^2\alpha^{1-\alpha}}{4\mathrm{e}^2}
	\Big(s-\frac{1-\alpha}{y(t)}\Big)^{(m_1-1)(\alpha-1)+\alpha-\frac{2}{n}}  \nonumber  \\
	&\,\left(		\frac{4\beta^{\beta}\mathrm{e}^{2}k_1 n^{m_1}l^{m_1-2}\alpha^{(m_1-1)(1-\alpha)-2+\frac{2}{n}}}
	{c_{2}^{\beta} }
	-{s}^{(m_1-1)(1-\alpha)+\beta-1+\frac{2}{n}} 
	\right). 
\end{align}
According to (\ref{sstar3}), we obtain that
\begin{align}\label{mid-1}
		\mathcal{P}[\underline{U}, \underline{W}](s, t)
			\leqslant&\,\frac{c_{2}^{\beta}\beta^{-\beta}nl^2\alpha^{1-\alpha}}{4\mathrm{e}^2}\Big(s-\frac{1-\alpha}{y(t)}\Big)^{\alpha-\delta} 		\left(\frac{4\beta^{\beta}\mathrm{e}^{2}\alpha^{\delta-1}}
		{c_{2}^{\beta} nl }
		-{s_{0}}^{\beta+\delta-1}\right) \nonumber \\   
		&\, +\frac{c_{2}^{\beta}\beta^{-\beta}nl^2\alpha^{1-\alpha}}{4\mathrm{e}^2}
		\Big(s-\frac{1-\alpha}{y(t)}\Big)^{(m_1-1)(\alpha-1)+\alpha-\frac{2}{n}}  \nonumber  \\
		&\, \left(		\frac{4\beta^{\beta}\mathrm{e}^{2}k_1 n^{m_1}l^{m_1-2}\alpha^{(m_1-1)(1-\alpha)-2+\frac{2}{n}}}
		{c_{2}^{\beta} }
		-{s_{0}}^{(m_1-1)(1-\alpha)+\beta-1+\frac{2}{n}} 
		\right) \nonumber \\
		\leqslant&\, 0.
\end{align}

Following a similar approach to that used in the proof of (\ref{mid-1}), we can also conclude that $\mathcal{Q}[\underline{U}, \underline{W}](s, t)
\leqslant 0$ for all $t \in (0,T) \cap \left(0, \frac{1}{\theta}\right)$ and $s \in \big(\frac{1}{y(t)},s_{0}\big]$ by (\ref{sstar1}), (\ref{sstar2}) and (\ref{sstar4}). 
\end{proof} 

%%%%%%%%%%%%%%%%%%%%%%%%%%%%%%%%%%%%%%%%%%%%%%%%%%%%%%%%%%%%%%%%%%
Finally, we prove $\mathcal{P}[\underline{U}, \underline{W}](s, t) \leqslant 0$ and $\mathcal{Q}[\underline{U}, \underline{W}](s, t) \leqslant 0$ in the outer region \( (s_{0}, R^n] \). To facilitate the proof of the subsequent lemma, we define
\begin{align}\label{1max}
	{D_1}_{\max} := \max\left\{ {D_1}(x) \mid x \in \left[nl\mathrm{e}^{-1} \alpha^{1 - \alpha} R^{n(\alpha - 1)}, nl   s^{\alpha - 1}_{0} \right]\right\}
\end{align} 
and 
\begin{align}\label{2max}
	{D_2}_{\max} := \max\left\{ {D_2}(x) \mid x \in \left[nl\mathrm{e}^{-1} \beta^{1 - \beta} R^{n(\beta - 1)}, nl   s^{\beta - 1}_{0} \right]\right\}.
\end{align}
Then, for a fixed $s_{0} < R^n$ satisfying  $(\ref{sstar1})$-$(\ref{sstar4})$, and $\alpha, \beta, \delta \in (0,1)$ as obtained in Lemma \ref{alphabeta}, we choose $\theta_0$ satisfying
\begin{align}\label{eq3_18}
		\theta_{0}  \cdot \frac{l {s_{0}}^\alpha}{\mathrm{e}} \geqslant ls^{\alpha-\delta}_{0} + \frac{n^2  R^{2n - 2} l {s_{0}}^{\alpha - 2}{D_1}_{\max}}{\alpha}  + \mu^{\star}l{s_{0}}^{\alpha - 1} R^n
	\end{align}
	and 
	\begin{align}\label{eq3_18-1}
		\theta_{0} \cdot \frac{l {s_{0}}^\beta}{\mathrm{e}} \geqslant l s^{\beta-\delta}_{0}+ \frac{n^2  R^{2n - 2}l {s_{0}}^{\beta - 2}{D_2}_{\max}}{\beta} \cdot  + \mu^{\star}l{s_{0}}^{\beta - 1} R^n.
	\end{align}
%%%%%%%%%%%%%%%%%%%%%%%%%%%%%%%%%%%%%%%%%%%%%%%%%%%%%%%%%%%%%%%%%%%%%%%%%
\begin{lem}\label{lem3.4}
Suppose that $m_1, m_2 > 1$ satisfy $(\ref{eq1.6})$. Let $\alpha, \beta, \delta \in (0,1)$ be as provided by Lemma \ref{alphabeta}, let $s_0$ satisfy $(\ref{sstar1})$-$(\ref{sstar4})$, and let $\theta_0$ satisfy $(\ref{eq3_18})$ and $(\ref{eq3_18-1})$. Assume that $T>0$ and $y(t) \in C^1([0, T))$satisfy
	\begin{align}\label{y3}
		\left\{\begin{array}{l}
			0\leqslant y^{\prime}(t) \leqslant y^{1+\delta}(t), \quad t\in (0,T), \\
			y(0) \geqslant y_0,
		\end{array}\right.
	\end{align}
	with $y_0>\frac{1}{s_{0}}$. Then, whenever $\theta>\theta_{0}$, the functions $\underline{U}$ and $\underline{W}$ from $(\ref{eq3_7})$ satisfy
	\begin{align*}
		\mathcal{P}[\underline{U}, \underline{W}](s, t) \leqslant 0, \quad \mathcal{Q}[\underline{U}, \underline{W}](s, t) \leqslant 0,
	\end{align*}
	for all $t \in (0,T) \cap \left(0, \frac{1}{\theta}\right)$ and $s \in \left(s_{0}, R^n\right)$.
\end{lem}
\begin{proof}
Due to $s>s_{0}>\frac{1}{y(t)}$, we have 
\begin{align}\label{s7}
	R^{n}>s-\frac{1-\alpha}{y(t)}>s_{0}-\frac{1-\alpha}{y(t)}>\alpha s_{0}
	, \quad R^{n}>s-\frac{1-\beta}{y(t)}>s_{0}-\frac{1-\beta}{y(t)}>\beta s_{0}.
\end{align}
By (\ref{s7}) and $\theta t<1$, we have
\begin{align*}
	n \mathrm{e}^{-\theta t}l \alpha^{1-\alpha}\Big(s-\frac{1-\alpha}{y(t)}\Big)^{\alpha-1} \in \left[nl \mathrm{e}^{-1} \alpha^{1 - \alpha} R^{n(\alpha - 1)}, nl s^{\alpha - 1}_{0} \right]
\end{align*}
and
\begin{align*}
	n \mathrm{e}^{-\theta t}l \beta^{1-\beta}\Big(s-\frac{1-\beta}{y(t)}\Big)^{\beta-1}
	\in \left[nl \mathrm{e}^{-1} \beta^{1 - \beta} R^{n(\beta - 1)}, nl s^{\beta - 1}_{0} \right].
\end{align*}
Combining this with (\ref{1max}) and (\ref{2max}) allows for the derivation of the following results
\begin{align}\label{fmaxgmax}
	D_1\Big(n\mathrm{e}^{-\theta t}l \alpha^{1-\alpha}\Big(s-\frac{1-\alpha}{y(t)}\Big)^{\alpha-1} \Big) \leqslant {D_1}_{\max},  
	\quad D_2\Big(n\mathrm{e}^{-\theta t}l \beta^{1-\beta}\Big(s-\frac{1-\beta}{y(t)}\Big)^{\beta-1} \Big) \leqslant {D_2}_{\max}.
\end{align}
According to the definitions of $\underline{U}$, $\underline{W}$ and $\mathcal{P}$, we have
\begin{align*}
	\begin{aligned}
		\mathcal{P}[\underline{U}, \underline{W}](s, t) =& \underline{U}_t-n^2 s^{2-\frac{2}{n}} \underline{U}_{s s}
		D_1\left(n \underline{U}_s\right)-n \underline{U}_s \cdot\Big(\underline{W}-\frac{\mu^{\star} s}{n}\Big) \\
		=& -\theta \mathrm{e}^{-\theta t} l \alpha^{-\alpha} \Big(s-\frac{1-\alpha}{y(t)}\Big) ^{\alpha} 
		+ \mathrm{e}^{-\theta t} l \alpha^{1-\alpha} (1-\alpha) \Big(s-\frac{1-\alpha}{y(t)}\Big)^{\alpha-1}\frac{y^{\prime}(t)}{y^2(t)} \\ \notag
		&+ \mathrm{e}^{-\theta t} n^2 s^{2-\frac{2}{n}} l \alpha^{1-\alpha} (1-\alpha) \Big(s-\frac{1-\alpha}{y(t)}\Big)^{\alpha-2}D_1\Big(n\mathrm{e}^{-\theta t}l \alpha^{1-\alpha}\Big(s-\frac{1-\alpha}{y(t)}\Big)^{\alpha-1} \Big) \\ \notag
		&n\mathrm{e}^{-\theta t}l \alpha^{1-\alpha}\Big(s-\frac{1-\alpha}{y(t)}\Big)^{\alpha-1}  \Big(\mathrm{e}^{-\theta t}l \beta^{-\beta} \Big(s-\frac{1-\beta}{y(t)}\Big)^{\beta}-\frac{\mu^{\star}s}{n}\Big).
	\end{aligned}
\end{align*}
Then by (\ref{eq3_18}) and the first results in (\ref{s7}) and (\ref{fmaxgmax}), combined with $s_{0}>\frac{1}{y(t)}$, we obtain
\begin{align*}
	\begin{aligned}
		\mathcal{P}[\underline{U}, \underline{W}](s, t) 
		\leqslant & -\frac{l\theta {s_{0}}^\alpha}{\mathrm{e}}+  l\alpha^{1-\alpha} (\alpha s_{0})^{\alpha-1} y^{\delta-1}(t) + \frac{n^2  R^{2n - 2} l {s_{0}}^{\alpha - 2}{D_1}_{\max}}{\alpha}  + \mu^{\star}l{s_{0}}^{\alpha - 1} R^n\\ \notag
		\leqslant & -\frac{l\theta {s_{0}}^\alpha}{\mathrm{e}}+  l s_{0}^{\alpha-\delta} + \frac{n^2  R^{2n - 2} l {s_{0}}^{\alpha - 2}{D_1}_{\max}}{\alpha}  + \mu^{\star}l{s_{0}}^{\alpha - 1} R^n\\ \notag
		\leqslant & \,0,
	\end{aligned}
\end{align*}
for all $t \in (0,T) \cap \left(0, \frac{1}{\theta}\right)$ and $s \in \left(s_{0}, R^n\right)$. Similarly, due to (\ref{eq3_18-1}) and the second restrictions in (\ref{s7}) and (\ref{fmaxgmax}), we have
$
		\mathcal{Q}[\underline{U}, \underline{W}](s, t) 
		\leqslant 0,
$
for all $t \in (0,T) \cap \left(0, \frac{1}{\theta}\right)$ and $s \in \left(s_{0}, R^n\right)$. This completes the proof.
\end{proof}

\begin{proof}
[Proof of Theorem \ref{thm1_1}]
We can easily check that
\begin{align}\label{s=0-all}
	\underline{U}(0, t) =U(0,t)= 0, 
	\quad \underline{W}(0, t) =W(0,t)= 0,
	\quad t \in [0, T).
\end{align}
Using the definition of $\underline{U}$ and (\ref{ldef}), along with $\alpha^{-\alpha}=\mathrm{e}^{-\alpha \ln \alpha} \leqslant e^{\frac{1}{\mathrm{e}}}$ and $\alpha<1$, for all $t\in(0,T)$, we have
\begin{align*}
	\underline{U}(R^n, t) 
	= \mathrm{e}^{-\theta t}l \alpha^{-\alpha}  \left( R^n - \frac{1 - \alpha}{y(t)} \right)^{\alpha} 
	\leqslant  \alpha^{-\alpha} R^{n\alpha} \frac{\mu_{\star} R^n}{n \mathrm{e}^{\frac{1}{\mathrm{e}}}(R^n + 1)} 
	\leqslant \frac{\mu_{\star} R^n}{n} \cdot \frac{R^{n\alpha}}{R^n + 1} 
	\leqslant \frac{\mu_{\star} R^n}{n}.
\end{align*}
Similarly, we find that
\begin{align*}
	\underline{W}(R^n, t) \leqslant \frac{\mu_{\star} R^n}{n},
	\quad t \in (0,T).
\end{align*}
Thus, combining above two estimates with (\ref{eq2.5}), we have
\begin{align}\label{s=R-all}
	\underline{U}(R^n, t) \leqslant \frac{\mu_{\star} R^n}{n} \leqslant U(R^n, t) , 
	\quad \underline{W}(R^n, t) \leqslant \frac{\mu_{\star} R^n}{n}\leqslant W(R^n, t),
	\quad t \in [0, T).
\end{align}
Denote
\begin{align*}
	\hat{M}_1(r) = \omega_n  \underline{U}(r^n, 0), \quad \quad \hat{M}_2(r) = \omega_n \underline{W}(r^n, 0),
	\quad  r \in [0, R].
\end{align*}
Here, $\omega_n$ denotes the surface area of the unit sphere. It follows from \eqref{eq1.8} that
\begin{align}\label{t=0-1}
	U(s, 0)  =\frac{1}{\omega_n} \int_{B_{s^{\frac{1}{n}}}(0)} u_0 \mathrm{d} x 
	\geqslant \frac{1}{\omega_n} \cdot \hat{M}_1\left(s^{\frac{1}{n}}\right)
	=\underline{U}(s, 0),
	\quad  s \in\left[0, R^n\right],
\end{align}
and 
\begin{align}\label{t=0-2}
	W(s, 0)  =\frac{1}{\omega_n} \int_{B_{s^{\frac{1}{n}}}(0)} w_0 \mathrm{d} x 
	\geqslant \frac{1}{\omega_n} \cdot \hat{M}_2\left(s^{\frac{1}{n}}\right)
	=\underline{W}(s, 0),
	\quad  s \in\left[0, R^n\right].
\end{align}

For fixed $l$ defined in $(\ref{ldef})$ and $\alpha, \beta ,\delta \in (0,1)$ satisfying \eqref{delta1}-\eqref{qb}, we choose 
	\begin{align}\label{kappa}
		\Lambda=\min\Big\{1,\frac{nl}{2\mathrm{e}^{2}(1-\alpha)}, \frac{nl}{2\mathrm{e}^{2}(1-\beta)}\Big\}.
	\end{align}
	Take $s_{0}$ satisfying (\ref{sstar1})-(\ref{sstar4}), $y_{\star}$ satisfying (\ref{eq3_8}) and $\theta>\theta_0$  given by \eqref{eq3_18} and \eqref{eq3_18-1}, then we set
	\begin{align}\label{y0}
		y_0>\max\Big\{ \big(\frac{\theta}{\delta \Lambda }\big)^{\frac{1}{\delta}}, y_{\star}, \frac{1}{s_{0}}\Big\}.
	\end{align}
and
\begin{align}\label{T}
	T = \frac{1}{\delta \Lambda} y_0^{-\delta}.
\end{align}
Let $y \in C^1([0, T))$ be the blow-up solution of
	\begin{align}\label{eq3.2}
		\begin{cases}
			y^{\prime}(t) = \Lambda y^{1+\delta}(t), \quad t \in (0, T), \\
			y(0) = y_0
		\end{cases}
	\end{align}
satisfying $y^{\prime}(t) \geqslant 0$ and $y(t)\rightarrow +\infty$ as $t \nearrow T$. Therefore, we can ensure that all the assumptions on \( y(t) \) stated in Lemmas~\ref{lem3.1}–\ref{lem3.4} are satisfied. Together with the parameter settings, the results of Lemmas~\ref{lem3.1}–\ref{lem3.4} establish that
\begin{align}\label{all}
	\mathcal{P}[\underline{U}, \underline{W}](s, t) \leqslant 0, 
	\quad \mathcal{Q}[\underline{U}, \underline{W}](s, t) \leqslant 0,
\end{align}
for all $ s \in (0, R^n)\backslash \big\{\frac{1}{y(t)}\big\}$ and $t \in (0, T)\cap \left(0, \frac{1}{\theta}\right)$. Since $y_0> \big(\frac{\theta}{\delta \Gamma }\big)^{\frac{1}{\delta}}$, we have $T= \frac{1}{\delta \Gamma} y_0^{-\delta}$, which ensures that $(0,T) \cap \left(0, \frac{1}{\theta}\right)=(0,T)$.

On the basis of results (\ref{t=0-2}) - (\ref{all}), and (\ref{eq2.5}), an application of the comparison principle Lemma \ref{lem2.1} yields the conclusion that
\begin{align*}
	\underline{U}(s, t) \leqslant U(s, t) , 
	\quad \underline{W}(s, t) \leqslant W(s, t),
	\quad  (s,t) \in [0, R^n]   \times [0, T).
\end{align*}
Combining this result with (\ref{s=0-all}) and the computation given in (\ref{phis}), we establish that
\begin{align}\label{us=0}
	\frac{1}{n} \cdot u(0, t) = U_s(0, t) \geqslant \underline{U}_s(0, t) = \mathrm{e}^{-\theta t} \cdot l y^{1 - \alpha}(t) \geqslant \frac{l}{\mathrm{e}} \cdot y^{1 - \alpha}(t) \rightarrow +\infty \quad \text{as } t \nearrow T.
\end{align}
Similarly, we conclude that
\begin{align}\label{ws=0}
	\frac{1}{n} \cdot w(0, t)  \geqslant \frac{l}{\mathrm{e}} \cdot y^{1 - \beta}(t) \rightarrow +\infty \quad \text{as } t \nearrow T.
\end{align}
Therefore, we consequently deduce that $T_{\max} \leqslant T < \infty$, which is a direct consequence of identities (\ref{us=0}) and (\ref{ws=0}), in conjunction with Proposition \ref{local}.

\end{proof}

\end{document}